\documentclass[11pt, a4paper]{article}

\usepackage[dvipsnames]{xcolor}
\usepackage[utf8]{inputenc}
\usepackage{amsmath, amsthm, amssymb}
\usepackage[margin=1in]{geometry} 
\usepackage{graphicx}
\usepackage[english]{babel}
\usepackage{hyperref} 
\hypersetup{
	colorlinks=true,
	linkcolor=blue,
	filecolor=magenta,      
	urlcolor=cyan,
	pdftitle={Noninvertibility and Bifurcation Phenomena in a Four-Partitions Piecewise Linear Map},
	pdfauthor={Wirot Tikjha},
}

\newtheorem{theorem}{Theorem}[section] 
\newtheorem{proposition}[theorem]{Proposition}
\newtheorem{lemma}[theorem]{Lemma}
\newtheorem{corollary}[theorem]{Corollary}
\newtheorem{definition}[theorem]{Definition}
\newtheorem{remark}[theorem]{Remark}
\newtheorem{example}{Example}[section]
\newcommand\blfootnote[1]{\begingroup\renewcommand\thefootnote{}\footnote{#1}\addtocounter{footnote}{-1}\endgroup}
\title{Multistability and Bifurcation Phenomena in a Four-Partitions Piecewise Linear Map}
\author{Wirot Tikjha \\ Pibilsongkram Rajabhat University, Thailand \\ E-mail:  wirottik@psru.ac.th}
\date{\today}

\begin{document}
	
	\maketitle
	
\begin{abstract}
	This paper investigates the global dynamics and bifurcation structures of a two-dimensional piecewise linear map defined by four partitions, involving absolute value terms for both state variables. We analyze the stability of fixed points and demonstrate that the emergence of period-2 dynamics is initiated by a Flip bifurcation. A detailed analysis of these cycles reveals a unique branch manifesting in two topologically distinct configurations, connected via a Border Collision Bifurcation. Furthermore, we show that higher-period cycles (period 3 and 4) emerge in stable-saddle pairs through Fold Border Collision Bifurcations, leading to regions of multistability organized by the stable manifolds of saddle orbits. Finally, we examine the degenerate parameter regime where the system reduces to a one-dimensional discontinuous map. In this case, we prove the existence of an absorbing invariant interval and identify parameter regions exhibiting robust chaotic dynamics.
\end{abstract}
	
	
	\section{Introduction}
	\blfootnote{This is the pre-print version of the manuscript. The final version has been accepted for publication in \textit{Journal of Difference Equations and Applications} \cite{Tikjha2026}.}
	Piecewise smooth (PWS) dynamical systems have received significant attention in recent decades due to their ability to model nonsmooth phenomena in diverse fields such as engineering, economics, and biology \cite{DiBernardo2008, Zhusubaliyev2003, Simpson2010}. Unlike smooth systems, the state space of a PWS map is partitioned into distinct regions, each governed by a different functional form. The boundaries separating these regions are known as switching manifolds or critical lines. The interaction of invariant sets with these boundaries leads to a rich variety of nonsmooth bifurcations, most notably the border collision bifurcation (BCB), a term originally introduced by Nusse and Yorke \cite{Nusse1992, Nusse1995}.
	
	In the context of two-dimensional piecewise linear (PWL) maps, the dynamics are often governed by the eigenvalues of the Jacobian matrices associated with each partition. While the local behavior within a partition is linear, the global dynamics can be highly complex, exhibiting chaos, multistability, and sophisticated bifurcation structures \cite{Banerjee1999, Glendinning2016}. A fundamental class of such systems is the two-dimensional border collision normal form, which has been extensively studied to understand phenomena such as the transition to chaos and center bifurcations \cite{Sushko2008, Simpson2008}.
	
	A critical property influencing the global dynamics of these maps is invertibility. The transition between invertible and non-invertible regimes is often marked by degenerate parameter values where the map reduces dimensionality, mapping a two-dimensional region onto a one-dimensional curve \cite{GardiniTikjha2020CSF}. Such transitions can lead to the emergence of complex attracting sets, including Milnor attractors and chaotic intervals \cite{GardiniTikjha2020IJBC}. Recent studies have also highlighted the existence of "weird" quasiperiodic attractors in discontinuous PWL maps, which challenge classical definitions of chaotic attractors \cite{Gardini2025arXiv, Gardini2025CSF}.
	
	In this work, we analyze a specific two-dimensional PWL map $T(x,y)$ defined by four linear partitions separated by the coordinate axes. This system belongs to a broader family of difference equations involving absolute value terms, which have been the subject of recent investigation \cite{Tikjha2010, Tikjha2017, GardiniTikjha2019}. While previous works often focused on maps with two partitions \cite{GardiniTikjha2020CSF}, the model considered here introduces higher complexity due to the presence of absolute values for both state variables, $|x|$ and $|y|$. In contrast to classical two-dimensional border collision normal forms and standard two-partition maps, which often exhibit specific structural rigidities, the inclusion of the $c|y|$ term breaks the invariant phase space structures common in simpler piecewise linear maps, generating a highly complex four-partition geometry. 
	
	The primary contribution of this paper is a rigorous analysis of the periodic orbits and bifurcation phenomena driven by this broken symmetry. We analytically determine the existence and stability regions of fixed points, 2-cycles, 3-cycles, and 4-cycles. Specifically, we demonstrate that this structural novelty is the fundamental mechanism allowing for the emergence of stable 2-cycles—a dynamical feature typically absent in two-partition counterparts—and leads to a novel topological organization of multistable basins governed by the stable manifolds of coexisting saddle cycles. Furthermore, we investigate the degenerate parameter case $c=a$, where the system exhibits dimension reduction, and the dynamics can be fully described by a one-dimensional first return map. This reduction allows us to prove the existence of chaotic intervals and describe the mechanisms driving global multistability.
	
	The paper is organized as follows: Section 2 provides the preliminaries and definitions of the map. Section 3 presents the analysis of periodic orbits, including fixed points and cycles of period 2, 3, and 4, along with their associated bifurcations. Section 4 investigates the degenerate case $c=a$ and derives the one-dimensional return map. Finally, we conclude with a summary of our findings.
	\section{Preliminaries}
	We consider the piecewise linear map $T: \mathbb{R}^2 \to \mathbb{R}^2$ defined by:
	\begin{equation}
		T(x,y) = (x', y') = (|x| + ay - 1, x + c|y| - 1)
	\end{equation}
	where $a$ and $c$ are real parameters. The phase plane $\mathbb{R}^2$ is partitioned into four regions $Q_i$ corresponding to the quadrants, with borders on the critical set $LC_{-1} = \{(x,y) \in \mathbb{R}^2 \mid x=0 \text{ or } y=0\}$.  Because the map is continuous across these axes, we can define the regions with inclusive, overlapping boundaries without ambiguity. The map $T$ is defined by four linear maps $T_i: Q_i \to \mathbb{R}^2$:
	\begin{itemize}
		\item \textbf{For $x \ge 0, y \ge 0$ (Q1):} $T_1(x,y) = (x + ay - 1, x + cy - 1)$
		\item \textbf{For $x \le 0, y \ge 0$ (Q2):} $T_2(x,y) = (-x + ay - 1, x + cy - 1)$
		\item \textbf{For $x \le 0, y \le 0$ (Q3):} $T_3(x,y) = (-x + ay - 1, x - cy - 1)$
		\item \textbf{For $x \ge 0, y \le 0$ (Q4):} $T_4(x,y) = (x + ay - 1, x - cy - 1)$
	\end{itemize}

In PWL maps, dynamics are governed by the \textbf{critical set} $LC_{-1}$ (where the map's definition changes) and its image, the \textbf{critical curve} $LC = T(LC_{-1})$. For map (1), $LC_{-1}$ is the union of the axes. The critical curve $LC$ is the union of four linear segments:
\begin{itemize}
	\item $LC1 = T_1(x=0, y>0): y = \frac{c}{a}x + \frac{c}{a} - 1$
	\item $LC2 = T_2(x=0, y>0): y = -\frac{c}{a}x - \frac{c}{a} - 1$
	\item $LC3 = T_1(x>0, y=0): y = x$
	\item $LC4 = T_2(x<0, y=0): y = -x - 2$
\end{itemize}
The map becomes dimension-reducing (non-invertible) when the Jacobian determinant in a partition is zero.
\begin{itemize}
	\item When $\mathbf{c = a}$: $\det(J_1) = c - a = 0$ and $\det(J_3) = c - a = 0$. Quadrants 1 and 3 collapse onto the lines $LC3$ and $LC4$, respectively.
	\item When $\mathbf{c = -a}$: $\det(J_4) = -c - a = 0$ and $\det(J_2) = -c - a = 0$. Quadrants 4 and 2 collapse onto the lines $LC3$ and $LC4$, respectively.
\end{itemize}

	\section{Analysis of Periodic Orbits and Bifurcations}
While the calculation of low-period cycles in standard smooth or two-partition systems is often straightforward, the inclusion of absolute value terms for both state variables creates a highly complex four-partition geometry. Consequently, the explicit analytical determination of these periodic orbits, their existence regions, and their stable manifolds is highly non-trivial and forms the essential foundation for mapping the global basin boundaries of the multistable attractors.
\subsection{Fixed Points and Virtual Fixed Points}
\begin{definition}[Real and Virtual Fixed Points]
	A fixed point $P_i$ is called a \textbf{real fixed point} if its coordinates lie within its defining quadrant, $P_i \in Q_i$. If $P_i \notin Q_i$, it is a \textbf{virtual fixed point}. The map $T$ can only have a fixed point at $P_i$ if $P_i$ is real.
\end{definition}
By solving the linear system $T_i(\mathbf{p}) = \mathbf{p}$ for $i \in \{1, 2, 3, 4\}$, we have the following proposition.
\begin{proposition}[Fixed Points and Virtual Fixed Points] \label{prop3.1}
	The map $T$ has four potential fixed points, $P_1= \left( \frac{a - c + 1}{a}, \frac{1}{a} \right), P_2 = \left( \frac{a - c + 1}{a + 2c - 2}, \frac{3}{a + 2c - 2} \right), P_3 = \left( \frac{a + c + 1}{a - 2c - 2}, \frac{3}{a - 2c - 2} \right), P_4 = \left( \frac{a + c + 1}{a}, \frac{1}{a} \right)$, corresponding to the fixed points of the linear maps $T_1, T_2, T_3, T_4$ respectively. 
\end{proposition}
\begin{lemma}[Jury Stability Conditions for 2D Maps \cite{Elaydi2024,Tikjha2025}] \label{lem:jury}
	Let $J$ be a $2 \times 2$ constant Jacobian matrix associated with a fixed point or a periodic cycle in a discrete dynamical system. The roots of the characteristic polynomial $\lambda^2 - Tr(J)\lambda + Det(J) = 0$ satisfy $|\lambda| < 1$ (indicating local asymptotic stability) if and only if the following three conditions hold simultaneously:
	\begin{enumerate}
		\item $1 - Tr(J) + Det(J) > 0$
		\item $1 + Tr(J) + Det(J) > 0$
		\item $|Det(J)| < 1$
	\end{enumerate}
\end{lemma}
The existence conditions ($P_i \in Q_i$) and stability conditions  reveals that the conditions are mutually exclusive for $P_1, P_2$, and $P_4$. Only the fixed point $P_3$, calculated in Proposition \ref{prop3.1}, is real (i.e., $P_3 \in Q_3$). So we have the following proposition.
	\begin{proposition}[Stability of the Real Fixed Point]\label{prop3.2}
	Of the four fixed points $P_1, P_2, P_3, P_4$, only $P_3$ can be simultaneously a real fixed point (exist in its quadrant $Q_3$) and be locally stable and stable in the non-empty parameter region defined by $a < 0$, $a > -c - 1$, and $|c - a| < 1$.
\end{proposition}
\begin{proof}
	We analyze the conditions for existence (i.e., $P_i \in Q_i$) and local stability for each of the four fixed points. A fixed point $P_i$ is locally stable if all eigenvalues $\lambda$ of its corresponding Jacobian $J_i$ satisfy $|\lambda| < 1$. This is true if and only if the Jury conditions are met.
	
	\subsubsection*{Analysis of $P_1$}
	\begin{itemize}
		\item \textbf{Fixed Point:} $P_1 = \left( \frac{a - c + 1}{a}, \frac{1}{a} \right)$
		\item \textbf{Existence ($P_1 \in Q_1$):} Requires $x \ge 0, y \ge 0$.
		$y = \frac{1}{a} \ge 0 \implies \mathbf{a > 0}$.
		\item \textbf{Stability (from $J_1 = \begin{pmatrix} 1 & a \\ 1 & c \end{pmatrix}$):}
		$\det(J_1) = c - a$ and $\text{Tr}(J_1) = 1 + c$.
		The second Jury condition is $1 - (1+c) + (c-a) > 0$, which simplifies to $-a > 0$, or $\mathbf{a < 0}$.
		\item \textbf{Conclusion:} The existence condition ($a > 0$) and the stability condition ($a < 0$) are mutually exclusive. Therefore, $P_1$ can never be both real and stable.
	\end{itemize}
	
	\subsubsection*{Analysis of $P_2$}
	\begin{itemize}
		\item \textbf{Fixed Point:} $P_2 = \left( \frac{a - c + 1}{a + 2c - 2}, \frac{3}{a + 2c - 2} \right)$
		\item \textbf{Existence ($P_2 \in Q_2$):} Requires $x \le 0, y \ge 0$.
		$y \ge 0 \implies a + 2c - 2 > 0 \implies \mathbf{c > 1 - a/2}$.
		\item \textbf{Stability (from $J_2 = \begin{pmatrix} -1 & a \\ 1 & c \end{pmatrix}$):}
		$\det(J_2) = -c - a$ and $\text{Tr}(J_2) = c - 1$.
		The second Jury condition is $1 - (c-1) + (-c-a) > 0$, which simplifies to $2 - 2c - a > 0$, or $\mathbf{c < 1 - a/2}$.
		\item \textbf{Conclusion:} The existence condition ($c > 1 - a/2$) and a necessary stability condition ($c < 1 - a/2$) are mutually exclusive. Therefore, $P_2$ can never be both real and stable.
	\end{itemize}
	
	\subsubsection*{Analysis of $P_3$}
	\begin{itemize}
		\item \textbf{Fixed Point:} $P_3 = \left( \frac{a + c + 1}{a - 2c - 2}, \frac{3}{a - 2c - 2} \right)$
		\item \textbf{Existence ($P_3 \in Q_3$):} Requires $x \le 0, y \le 0$.
		\begin{enumerate}
			\item $y \le 0 \implies a - 2c - 2 < 0 \implies \mathbf{c > a/2 - 1}$.
			\item $x \le 0 \implies \frac{a + c + 1}{a - 2c - 2} \le 0$. Since the denominator is negative, this requires $a + c + 1 \ge 0 \implies \mathbf{c \ge -a - 1}$.
		\end{enumerate}
		\item \textbf{Stability (from $J_3 = \begin{pmatrix} -1 & a \\ 1 & -c \end{pmatrix}$):}
		$\det(J_3) = c - a$ and $\text{Tr}(J_3) = -1 - c$.
		\begin{enumerate}
			\item $|\det(J_3)| < 1 \implies \mathbf{|c - a| < 1}$.
			\item $1 - \text{Tr}(J_3) + \det(J_3) > 0 \implies 1 - (-1-c) + (c-a) > 0 \implies 2 + 2c - a > 0 \implies \mathbf{c > a/2 - 1}$.
			\item $1 + \text{Tr}(J_3) + \det(J_3) > 0 \implies 1 + (-1-c) + (c-a) > 0 \implies -a > 0 \implies \mathbf{a < 0}$.
		\end{enumerate}
		\item \textbf{Conclusion:} We must satisfy all conditions.
		The stability condition $c > a/2 - 1$ is identical to existence condition (1).
		We must check if the existence condition $c \ge -a - 1$ is compatible with the stability condition $c > a/2 - 1$. For $a < 0$, we have $-a > a/2$, which implies $-a-1 > a/2 - 1$. Thus, the existence condition $c \ge -a - 1$ is stricter and automatically satisfies $c > a/2 - 1$.
		
		Therefore, the complete set of conditions for $P_3$ to be both real and stable is:
		$\mathbf{a < 0}$, $\mathbf{c \ge -a - 1}$, and $\mathbf{|c - a| < 1}$.
		This is the region specified in the proposition (using the strict inequality $a > -c-1$ to define the open region of stability, excluding the border $x=0$). 
	\end{itemize}
	
	\subsubsection*{Analysis of $P_4$}
	\begin{itemize}
		\item \textbf{Fixed Point:} $P_4 = \left( \frac{a + c + 1}{a}, \frac{1}{a} \right)$
		\item \textbf{Existence ($P_4 \in Q_4$):} Requires $x \ge 0, y < 0$.
		$y < 0 \implies \mathbf{a < 0}$.
		$x \ge 0 \implies \frac{a + c + 1}{a} \ge 0$. Since the denominator is negative, this requires $a + c + 1 \le 0 \implies \mathbf{c \le -a - 1}$.
		\item \textbf{Stability (from $J_4 = \begin{pmatrix} 1 & a \\ 1 & -c \end{pmatrix}$):}
		$\det(J_4) = -c - a$ and $\text{Tr}(J_4) = 1 - c$.
		The stability condition $|\det(J_4)| < 1$ requires $|-c - a| < 1$, or $\mathbf{|a + c| < 1}$. This implies $-1 - a < c < 1 - a$. The lower bound is $\mathbf{c > -a - 1}$.
		\item \textbf{Conclusion:} The existence condition ($c \le -a - 1$) and the stability condition ($c > -a - 1$) are mutually exclusive. Therefore, $P_4$ can never be both real and stable.
	\end{itemize}
	Only $P_3$ can be simultaneously real and stable, which occurs in the non-empty region derived.
\end{proof}

The stability region for the real fixed point $P_3$, shown as the yellow region in the $(a, c)$ parameter plane in Fig. \ref{FP2Cy}, is defined by three critical boundaries, each corresponding to a specific bifurcation. The condition $\mathbf{a < 0}$ is bounded by the line $\mathbf{a = 0}$, which is a Flip bifurcation where an eigenvalue passes through $\lambda = -1$, typically creating a 2-cycle. The condition $\mathbf{c \ge -a - 1}$ ensures the fixed point's existence within $Q_3$; its boundary $\mathbf{c = -a - 1}$ is a Border Collision Bifurcation (BCB) where the fixed point $P_3$ collides with the y-axis (the border of $Q_3$) and becomes virtual. Finally, the stability condition $\mathbf{|c - a| < 1}$ (from $|\det(J_3)| < 1$) is bounded by two distinct curves: $\mathbf{c - a = 1}$, which is a Center bifurcation where $P_3$ loses stability to an invariant closed curve, and $\mathbf{c - a = -1}$, which is a determinant $ det (J_3) = -1$ bifurcation that also destabilizes the fixed point.
\begin{figure}[h]
	\centering
	\includegraphics[scale=0.75]{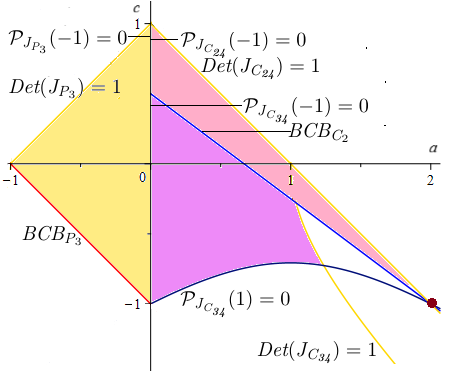}
	\caption{The (a, c) parameter plane, showing 2D bifurcation curves.}\label{FP2Cy}
\end{figure}
\begin{remark}[Center Bifurcation as a Degenerate Bifurcation]
	In the context of this piecewise linear system, the center bifurcation acts as a degenerate Neimark-Sacker bifurcation. Because the mapping is piecewise linear, classical bifurcations become degenerate, combining the properties of smooth bifurcations with border-collision mechanisms. Consequently, when complex conjugate eigenvalues cross the unit circle ($|Det(J)| = 1$), the system loses stability and transitions instantly to invariant closed curves or polygons, lacking the gradual amplitude scaling characteristic of a classical Neimark-Sacker bifurcation.
\end{remark}
\subsection{2-Cycles}

	Unlike the 2-partition map $T(x,y) = (|x| - ay, x - cy + 1)$ \cite{GardiniTikjha2020CSF}, this 4-partition map supports 2-cycles due to the $c|y|$ term, which breaks the structural rigidity that previously prevented 2-cycles. The emergence of period-2 dynamics in this map is intrinsically linked to the stability properties of the fixed point $P_3$. As established in Proposition \ref{prop3.2}, $P_3$ loses stability via a Flip bifurcation at the boundary $a=0$, where an eigenvalue passes through $-1$.
	
	From a global bifurcation perspective, the 2-cycles described below should be considered as a \textbf{unique period-2 branch} emerging from this Flip bifurcation. Depending on the value of the parameter $c$ (specifically, the orientation of the eigenvector associated with $\lambda = -1$ relative to the discontinuity boundaries), this branch shows in different partitions. The two resulting configurations, denoted as $C_{34}$ and $C_{24}$, occupy adjacent regions in the parameter space and are separated by a Border Collision Bifurcation (BCB). Consequently, they never coexist; rather, the cycle undergoes a topological change in its symbolic sequence as it crosses the critical boundary.
	
	\begin{proposition}[Existence and Stability of 2-Cycles]
		The map $T$ supports a period-2 branch that manifests as two distinct configurations depending on the parameter region:
		\begin{enumerate}
			\item \textbf{Cycle $C_{34}$:} A 2-cycle $\{\mathbf{p_1}, \mathbf{p_2}\}$ switching between Quadrant 4 ($T_4$) and Quadrant 3 ($T_3$), defined when the cycle avoids Quadrant 2.
			\item \textbf{Cycle $C_{24}$:} A 2-cycle $\{\mathbf{q_1}, \mathbf{q_2}\}$ switching between Quadrant 2 ($T_2$) and Quadrant 4 ($T_4$), defined when the cycle avoids Quadrant 3.
		\end{enumerate}
	\end{proposition}
	
	\begin{remark}[Mechanism of Cycle Selection]
		The selection between $C_{34}$ and $C_{24}$ at the moment of the Flip bifurcation ($a=0$) is determined by the geometry of the eigenspace. Following the eigenvector of the fixed point $P_3$ associated with the eigenvalue $\lambda = -1$, the generated 2-cycle is created along the direction of this vector. When this direction intersects the discontinuity $x=0$ or $y=0$, it dictates the symbolic sequence of the nascent cycle.
	\end{remark}
	
\begin{proposition}[Existence and Stability of 2-Cycles]
	The map $T$ supports (at least) two distinct 2-cycles, $C_{34}$ and $C_{24}$.
	\begin{enumerate}
		\item \textbf{Cycle $C_{34}$:} There exists a 2-cycle $\{\mathbf{p_1}, \mathbf{p_2}\}$ switching between Quadrant 4 ($T_4$) and Quadrant 3 ($T_3$), given by $\mathbf{p_1} \in Q_4$, $\mathbf{p_2} \in Q_3$ where $\mathbf{p_2} = T_4(\mathbf{p_1})$ and $\mathbf{p_1} = T_3(\mathbf{p_2})$.
		\item \textbf{Cycle $C_{24}$:} There exists a 2-cycle $\{\mathbf{q_1}, \mathbf{q_2}\}$ switching between Quadrant 2 ($T_2$) and Quadrant 4 ($T_4$), given by $\mathbf{q_1} \in Q_2$, $\mathbf{q_2} \in Q_4$ where $\mathbf{q_2} = T_2(\mathbf{q_1})$ and $\mathbf{q_1} = T_4(\mathbf{q_2})$.
	\end{enumerate}
\end{proposition}

\begin{proof}
	We prove the existence and stability for each cycle separately.
	
	\subsubsection*{Case 1: 2-Cycle $C_{34}$}
	The coordinates are found by solving the linear system $\mathbf{p_2} = T_4(\mathbf{p_1})$ and $\mathbf{p_1} = T_3(\mathbf{p_2})$. This calculation yields:
	\begin{align*}
		\mathbf{p_1} &= \left( \frac{a(a + c + 1)}{D_{34}}, \frac{a + 2c - 4}{D_{34}} \right) \\
		\mathbf{p_2} &= \left( \frac{(a+2c-2)(a+c+1)}{D_{34}}, \frac{3a + 4c - 2}{D_{34}} \right)
	\end{align*}
	with the common denominator $D_{34} = a^2 - 2a - 2c^2 + 2$. Existence as a real cycle requires $\mathbf{p_1} \in Q_4$  and $\mathbf{p_2} \in Q_3$.
	
	Stability is determined by the eigenvalues of the Jacobian $J = J_3 \cdot J_4$. The stability conditions $\lvert \lambda(J) \rvert < 1$ are met if and only if:
	\begin{enumerate}
		\item $a^2 - 2a - 2c^2 + 2 > 0$
		\item $a(a + 2) > 0$
		\item $|a^2 - c^2| < 1$
	\end{enumerate}
	
	\subsubsection*{Case 2: 2-Cycle $C_{24}$}
	The coordinates are found by solving the linear system $\mathbf{q_2} = T_2(\mathbf{q_1})$ and $\mathbf{q_1} = T_4(\mathbf{q_2})$, which yields:
	\begin{align*}
		\mathbf{q_1} &= \left( \frac{a^2 - ac - 2c^2 - a - 2}{D_{24}}, \frac{3a + 4c - 2}{D_{24}} \right) \\
		\mathbf{q_2} &= \left( \frac{a(a + 3c + 1)}{D_{24}}, \frac{a - 2c - 4}{D_{24}} \right)
	\end{align*}
	with the common denominator $D_{24} = a^2 - 2a + 2ac + 2c^2 + 2$. Existence requires $\mathbf{q_1} \in Q_2$  and $\mathbf{q_2} \in Q_4$.
	
	Stability is determined by the eigenvalues of the Jacobian $J = J_4 \cdot J_2$. The stability conditions $\lvert \lambda(J) \rvert < 1$ are met if and only if:
	\begin{enumerate}
		\item $(a + c)^2 < 1$
		\item $|2a - c^2 - 1| < 1 + (a+c)^2$
	\end{enumerate}
\end{proof}

The stability regions of the two 2-cycles, $C_{34}$ (below the blue line) and $C_{24}$ (above the blue line), which are shown as two different shaded pink regions in Fig. \ref{FP2Cy}, are bounded by several bifurcation curves, and the cycles themselves are linked by a border collision. For $C_{34}$, stability is lost at the center bifurcation on the curve $|a^2 - c^2| = 1$, where the cycle is replaced by an invariant closed curve, and at the flip bifurcations on $a(a + 2) = 0$. Similarly, the $C_{24}$ cycle loses stability at its own center bifurcation on the boundary $(a + c)^2 = 1$ and at fold or flip bifurcations on the curves $|2a - c^2 - 1| = 1 + (a+c)^2$. Critically, the two cycles are directly related by a border collision bifurcation on the blue line $3a + 4c - 2 = 0$. On this curve, the $y$-coordinate of point $\mathbf{p_2}$ (in $C_{34}$) and point $\mathbf{q_1}$ (in $C_{24}$) becomes zero, causing the point to collide with the x-axis (the border $y=0$). This collision transforms the $C_{34}$ cycle (sequence 4-3) into the $C_{24}$ cycle (sequence 4-2), or vice-versa.
\begin{figure}[h]
	\centering
	\includegraphics[scale=0.6]{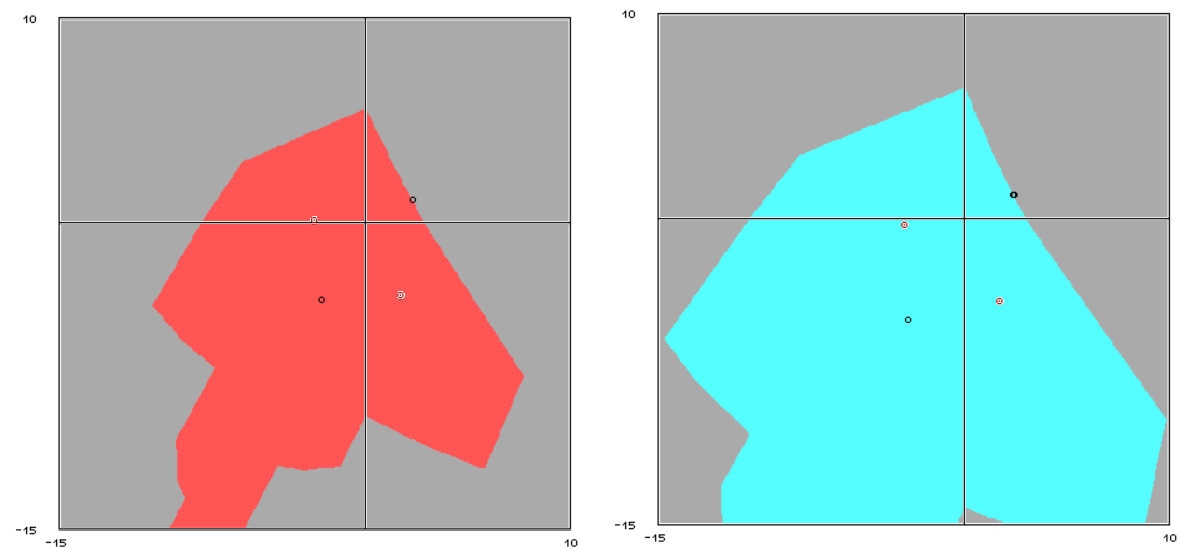}
	\caption{For $a = 0.9$, cycle dynamics are shown from left to right: a 2-cycle with symbolic sequence \textit{24} at $c = -0.15$ and a cycle with symbolic sequence \textit{34} at $c = -0.25$.}\label{2Cy}
\end{figure}
In Fig. \ref{2Cy},  the left panel shows the 2-cycle with symbolic sequence \textit{24} and its basin of attraction, shaded in red. The unstable fixed points $P_1$ and $P_3$ are also shown, marked by black circles. Moreover, the fixed point $P_1$ is observed to lie on the \textbf{boundary} of the 2-cycle's basin of attraction. Similarly, the right panel shows the 2-cycle with symbolic sequence \textit{34} and its basin of attraction, shaded in azure. In this case as well, the fixed point $P_1$ is shown on the boundary of the 2-cycle's basin.

 We investigate the hypothesis that the fixed point $P_1$ acts as a saddle point that organizes the basin of attraction for the 2-cycles. First, we find the general properties of $P_1$: its coordinates are $P_1 = \left( \frac{a - c + 1}{a}, \frac{1}{a} \right)$ and its eigenvalues are $\lambda_{1,2} = \frac{1+c \pm \sqrt{(c-1)^2 + 4a}}{2}$.
 
 Substituting the parameters for the $C_{24}$ cycle at $\mathbf{a=0.9, c=-0.15}$, we obtain $\mathbf{\lambda_1 \approx 1.534}$ ($|\lambda_1|>1$) and $\mathbf{\lambda_2 \approx -0.684}$ ($|\lambda_2|<1$). Here, $P_1$ is indeed a \textbf{saddle point}, and its stable manifold $W^s(P_1)$ forms the basin boundary for the attractor.
 
However, the role of $P_1$ is not permanent. As parameters evolve (e.g., increasing $a$), $P_1$ may undergo a bifurcation. Specifically, for the parameters used in the center bifurcation analysis shown in Fig. \ref{2Center} ($\mathbf{a=1.1}$), the eigenvalues of $P_1$ become both greater than 1 in modulus (or complex outside the unit circle), transforming $P_1$ into a \textbf{repelling node}.

As pointed out in the bifurcation structure, when $P_1$ loses its saddle stability, a new saddle orbit is typically created on the boundary to take over the role of separating the basins. In this system, this role is assumed by the saddle 2-cycle $C_{14}$ (symbolic sequence $Q_1 \to Q_4$), which is created via a bifurcation involving the discontinuity boundaries. As we will show in the subsequent proposition, for $a=1.1$, the cycle $C_{14}$ exists and is a saddle, and its stable set $W^s(C_{14})$ forms the new basin boundary visible in Fig. \ref{2Center}.

\begin{figure}[h]
	\centering
	\includegraphics[scale=0.6]{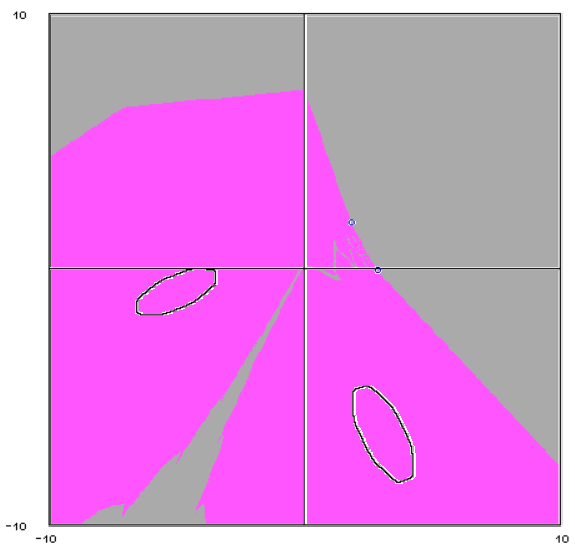}
	\caption{The figure illustrates a center bifurcation for $a = 1.1$.  After the bifurcation ($c = -0.457$), the 2-cycles $C_{34}$ are replaced by an invariant closed curve.}\label{2Center}
\end{figure}
In addition to the border-collision and flip bifurcations, the period-2 cycle also loses stability via a Center bifurcation. As predicted by the eigenvalue analysis (when the eigenvalues become complex conjugates on the unit circle), the stable 2-cycle is destroyed and gives rise to an attracting invariant closed curve. This phenomenon is illustrated numerically in Fig.~\ref{2Center}. Specifically for the configuration $C_{34}$ shown in Fig.~\ref{2Center} with a fixed $\mathbf{a=1.1}$, the two stable periodic points present at $\mathbf{c = -0.46}$ (prior to the bifurcation at $\mathbf{c \approx - 0.4582575695}$) are replaced by an invariant closed curve after the bifurcation at $\mathbf{c = -0.457}$.

We also investigate a 2-cycle that is numerically observed to lie on the
boundary of the basin of attraction. This cycle, which we denote $C_{14}$, 
possesses the symbolic sequence $Q_1 \to Q_4$. To analyze its properties
and confirm its role as a saddle cycle, we now calculate its coordinates
and eigenvalues. As the subsequent analysis will show, the characteristic 
polynomial for $C_{14}$ is distinct from those of $C_{34}$ and $C_{24}$, 
resulting in different eigenvalues.

\begin{proposition}[Coordinates of 2-Cycle $C_{14}$]
There exists a 2-cycle $\{\mathbf{p_1}, \mathbf{p_2}\}$ switching between Quadrant 1 ($T_1$) and Quadrant 4 ($T_4$), given by $\mathbf{p_1} \in Q_1$, $\mathbf{p_2} \in Q_4$ where $\mathbf{p_2} = T_1(\mathbf{p_1})$ and $\mathbf{p_1} = T_4(\mathbf{p_2})$.
\end{proposition}

\begin{proof}
We solve the linear system $\mathbf{p_2} = T_1(\mathbf{p_1})$ and $\mathbf{p_1} = T_4(\mathbf{p_2})$.
Let $\mathbf{p_1} = (x_1, y_1)$ and $\mathbf{p_2} = (x_2, y_2)$. The system is:
\begin{align*}
x_2 &= x_1 + ay_1 - 1 \\
y_2 &= x_1 + cy_1 - 1 \\
x_1 &= x_2 + ay_2 - 1 \\
y_1 &= x_2 - cy_2 - 1
\end{align*}
Solving this system yields the coordinates for the periodic points:
\begin{align*}
\mathbf{p_1} = (x_1, y_1) &= \left( \frac{2 + a(1+c-a) + 2c^2}{a(2 - a)}, \frac{2 - 2c - a}{a(2 - a)} \right) \\
\mathbf{p_2} = (x_2, y_2) &= \left( \frac{2 + a(1-c-a) + 2c^2}{a(2 - a)}, \frac{2 + 2c - a}{a(2 - a)} \right)
\end{align*}
with the common denominator $D = a(2 - a)$.
\end{proof}

\begin{proposition}[Eigenvalues of 2-Cycle $C_{14}$]
The local stability of the 2-cycle $C_{14}$ is determined by the eigenvalues $(\lambda_1, \lambda_2)$ of the Jacobian $J_{14} = J_4 \cdot J_1$.
\end{proposition}

\begin{proof}
The Jacobian of the 2-iterate map is $J_{14} = J_4 \cdot J_1$.
$$J_{14} = \begin{pmatrix} 1 & a \\ 1 & -c \end{pmatrix} \begin{pmatrix} 1 & a \\ 1 & c \end{pmatrix} = \begin{pmatrix} 1+a & a+ac \\ 1-c & a-c^2 \end{pmatrix}$$
The eigenvalues $\lambda$ are the roots of the characteristic polynomial $p(\lambda) = \lambda^2 - \text{Tr}(J_{14})\lambda + \det(J_{14}) = 0$.

\begin{itemize}
\item \textbf{Determinant:}
$$\det(J_{14}) = \det(J_4) \cdot \det(J_1) = (-c-a)(c-a) = \mathbf{a^2 - c^2}$$

\item \textbf{Trace:}
$$\text{Tr}(J_{14}) = (1+a) + (a-c^2) = \mathbf{1 + 2a - c^2}$$

\end{itemize}
Therefore, the eigenvalues are the roots of the equation:
$$\lambda^2 - (1 + 2a - c^2)\lambda + (a^2 - c^2) = 0$$
Using the quadratic formula, the eigenvalues are:
$$\lambda_{1,2} = \frac{(1 + 2a - c^2) \pm \sqrt{(1 + 2a - c^2)^2 - 4(a^2 - c^2)}}{2}$$
\end{proof}

The calculation of the 2-cycle $C_{14}$ at the parameters $a = 1.1$ and $c = -0.457$ provides a critical insight into the global dynamics of the system. These parameters are chosen to be just after the Center  bifurcation of the $C_{34}$ cycle, where the attractor is no longer the cycle itself but the new invariant closed curve. A fundamental question for any attractor is what forms its basin of attraction, which must be the stable manifold of a coexisting saddle-type set. Our calculations confirm that at these parameters, the 2-cycle $C_{14}$ (with symbolic sequence $Q_1 \to Q_4$) is both real, as its coordinates $\mathbf{p_1} \approx (1.82, 1.83)$ and $\mathbf{p_2} \approx (2.84, -0.014)$ lie within their respective quadrants, and a saddle cycle, as its eigenvalues are $\mathbf{\lambda_1 \approx 2.607}$ ($|\lambda_1| > 1$) and $\mathbf{\lambda_2 \approx 0.384}$ ($|\lambda_2| < 1$). This confirms the mechanism for the global dynamics: the basin of attraction for the invariant closed curve (which was born from the bifurcation of $C_{34}$) is defined by the stable manifold, $W^s(C_{14})$, of this coexisting saddle 2-cycle.
\subsection{Analysis of the 3-Cycles}

In continuous piecewise linear maps, periodic orbits that do not originate from period-doubling (flip) bifurcations generally emerge via a \textbf{Fold Border Collision Bifurcation (Fold-BCB)}. This mechanism creates cycles in pairs—typically one stable (node or focus) and one unstable (saddle)—whose symbolic sequences differ by exactly one symbol. This difference corresponds to the specific periodic point that collides with a discontinuity boundary at the moment of bifurcation.

Consistent with this theory, we observe the simultaneous birth of a stable 3-cycle, denoted $C_{434}$ (sequence $4 \to 3 \to 4$), and a saddle 3-cycle, denoted $C_{334}$ (sequence $3 \to 3 \to 4$). The symbolic difference in the first iterate ($4$ vs $3$) indicates that this pair is created when a periodic point collides with the boundary separating Quadrant 4 and Quadrant 3 (the axis $x=0$).

The specific Fold-BCB curve leading to the creation of this pair is defined by the condition where the relevant coordinate vanishes. As we will derive in the coordinate analysis, this bifurcation boundary is given by the equation:
\begin{equation} \label{eq:3cycBCB}
	a^2 + ac - c^2 + c - 1 = 0
\end{equation}
We first detail the properties of the attracting branch of this pair.
\begin{proposition}[Existence of the 3-Cycle $C_{434}$]
	There exists a 3-cycle with symbolic sequence $4 \rightarrow 3 \rightarrow 4$, consisting of three distinct periodic points, $\{p_1, p_2, p_3\}$, where $p_1 \in Q_4$, $p_2 \in Q_3$, and $p_3 \in Q_4$. The explicit rational coordinates for these periodic points, which share the common denominator $D_p(a,c) = a^3 + a^2c - ac^2 - 2c^3 - 3ac + a - 2$, are provided in Appendix \ref{app:coordinates}.
\end{proposition}

\begin{proposition}[Stability of the 3-Cycle $C_{434}$]
	The 3-cycle with symbolic sequence $4 \to 3 \to 4$ is stable if and only if the parameters $(a,c)$ satisfy the following three conditions:
	\begin{enumerate}
		\item $ |(a-c)(a+c)^2| < 1 $
		\item $ -a^{3}-a^{2} c+\left(c^{2}+3 c-1\right) a+2 c^{3}+2 > 0 $
		\item $ a^{2}+a c-c^{2}+3 c-1 > 0 $
	\end{enumerate}
\end{proposition}

The stability and existence regions for this 3-cycle are shown in the $(a,c)$ parameter plane in Fig.~\ref{fig:3cycle_sta}. The brown region in Fig.~\ref{fig:3cycle_sta} is the stability region of the 3-cycle, which coexists with the fixed point in yellow region. This stability region is bounded by several curves, including flip bifurcation in blue curve and two distinct border collision bifurcations (BCBs) associated with the point $\mathbf{p_1}$ and $\mathbf{p_3}$:
\begin{itemize}
	\item The blue curve is the BCB for the collision with the y-axis ($N_{xp3} = 0$):
	$$ \left(a+c+1\right) \left(a^{2}+a c-c^{2}+c-1\right) = 0 $$
	\item The red dashed line is the BCB for the collision with the x-axis ($N_{yp1} = 0$):
	$$ a^2 + (3c-2)a + 3c^2 - 5c + 1 = 0 $$
\end{itemize}
\begin{figure}[h!]
	\centering
	\includegraphics[scale=0.4]{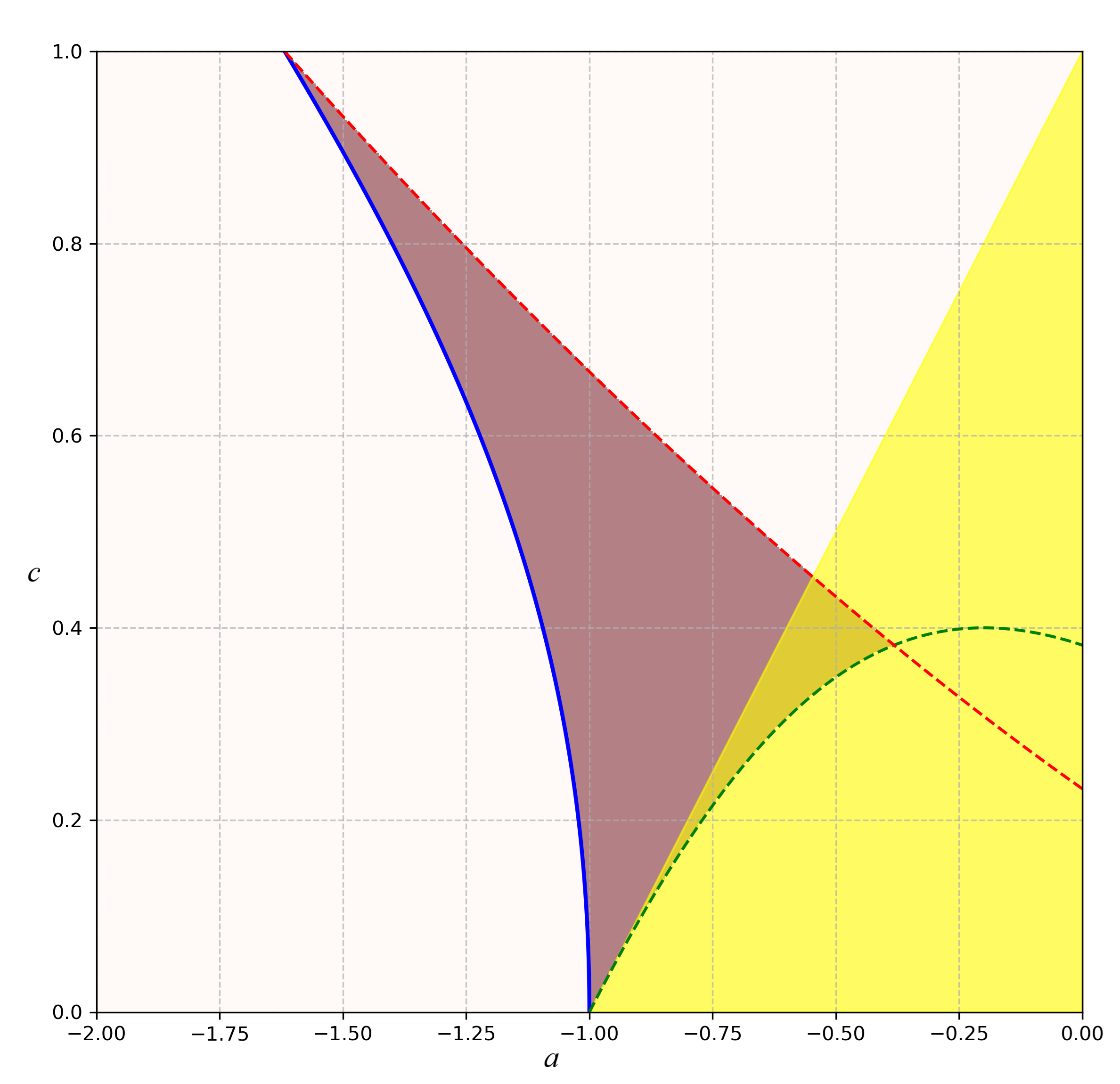} 
	\caption{The brown region is the stability region of the 3-cycle, shown coexisting with the fixed point. The region is bounded by two Border Collision Bifurcation (BCB) curves of the point $\mathbf{p_1}$: the collision with the y-axis ($N_{xp3} = 0$, blue curve) and the collision with the x-axis ($N_{yp1} = 0$, red dashed line).}
	\label{fig:3cycle_sta}
\end{figure}
The 3-cycle undergoes two primary types of bifurcations at the boundaries of its stability region: a flip bifurcation and border collision bifurcations. The cycle loses stability when Condition 3 is violated ($a^{2}+a c-c^{2}+3 c-1 = 0$ green dashed curve in Fig. \ref{fig:3cycle_sta}), constituting a flip bifurcation where an eigenvalue passes through $-1$; this leads to the birth of a stable 6-cycle (symbolic sequence $4 \to 3 \to 4 \to 4 \to 3 \to 4$) via period doubling, as illustrated in Fig.~\ref{fig:3cycle_flip}. 

\begin{figure}[h!]
	\centering
	\includegraphics[scale=0.6]{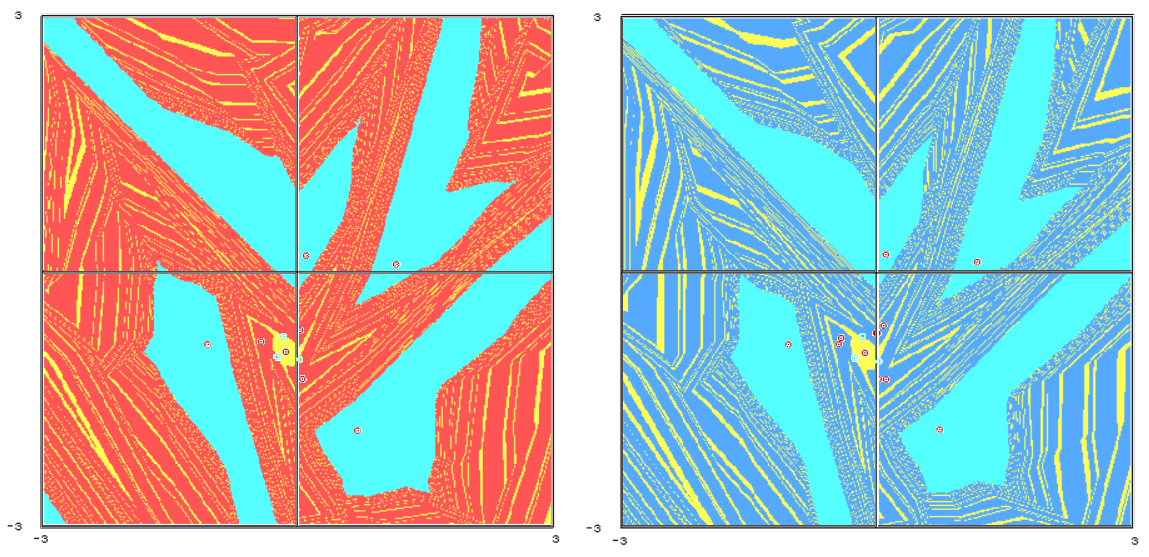} 
\caption{Flip bifurcation of the 3-cycle for fixed $c=0.2$. The panels illustrate the dynamics before the bifurcation at $a = -0.78$ (left) and after the bifurcation at $a = -0.77$ (right), in the presence of a coexisting 4-cycle attractor. The basins of attraction are color-coded: the yellow region corresponds to the fixed point $P_3$, the red region to the 3-cycle, the light blue region to the 6-cycle, and the azure region to the 4-cycle.}
	\label{fig:3cycle_flip}
\end{figure}

The phase portrait presented in the left panel of Fig.~\ref{fig:3cycle_flip} provides numerical evidence for the existence of a saddle 3-cycle, marked by azure circles, located on the boundary separating the basin of attraction of the stable 3-cycle (shaded in red) from that of the fixed point $P_3$ (shaded in yellow). The analytical coordinates of this saddle orbit are established in the following proposition.
\begin{proposition}[Existence of Saddle 3-Cycle $C_{334}$]
	There exists a 3-cycle with symbolic sequence $3 \rightarrow 3 \rightarrow 4$, denoted as $C_{334} = \{\mathbf{s_1}, \mathbf{s_2}, \mathbf{s_3}\}$, where $\mathbf{s_1} \in Q_3$, $\mathbf{s_2} \in Q_3$, and $\mathbf{s_3} \in Q_4$. The explicit rational coordinates for these periodic points, which share the common denominator $D_s(a,c) = a(a^2 - ac - c^2 - 3c - 1)$, are provided in Appendix \ref{app:coordinates}.
\end{proposition}

\begin{proof}
	The coordinates are derived by solving the linear system corresponding to the symbolic sequence $T_3 \to T_3 \to T_4$. Let $\mathbf{s_1} = (x_1, y_1)$. The system is defined by $\mathbf{s_2} = T_3(\mathbf{s_1})$, $\mathbf{s_3} = T_3(\mathbf{s_2})$, and $\mathbf{s_1} = T_4(\mathbf{s_3})$:
	\begin{align*}
		x_2 &= -x_1 + ay_1 - 1, & y_2 &= x_1 - cy_1 - 1 \\
		x_3 &= -x_2 + ay_2 - 1, & y_3 &= x_2 - cy_2 - 1 \\
		x_1 &= x_3 + ay_3 - 1, & y_1 &= x_3 - cy_3 - 1
	\end{align*}
	Solving this linear system for $(x_1, y_1, x_2, y_2, x_3, y_3)$ yields the rational expressions provided in the proposition.
	
	The stability of this cycle is determined by the eigenvalues of the Jacobian matrix $J_{334} = J_4 \cdot J_3 \cdot J_3$. In the region of existence where the stable 3-cycle $C_{434}$ exists (e.g., $a \approx -0.78, c=0.2$), the eigenvalues of $J_{334}$ are real with $|\lambda_1| > 1$ and $|\lambda_2| < 1$, confirming that $C_{334}$ is a saddle cycle.
\end{proof}

\begin{remark}[Relation to Border Collision Bifurcation]
	It is significant to note that the numerator of the $x$-coordinate for point $\mathbf{s_2}$ is:
	\[ N_{xs2} = (a + c + 1)(a^2 + ac - c^2 + c - 1) \]
	The factor $(a^2 + ac - c^2 + c - 1)$ is identical to the defining equation of the Border Collision Bifurcation (BCB) curve discussed in the previous section. This confirms that the saddle cycle $C_{334}$ and the attracting cycle $C_{434}$ are a complementary pair born simultaneously at this fold-BCB. At the bifurcation point, $\mathbf{s_2}$ (of the saddle) and $\mathbf{p_3}$ (of the attractor) collide at the border $x=0$.
\end{remark}
As shown in Fig.~\ref{fig:3cycle_flip}, as the parameter $a$ increases from $-0.78$ to $-0.77$ (for fixed $c=0.2$), the stable 3-cycle $C_{434}$ undergoes a supercritical flip (period-doubling) bifurcation at the critical value $a \approx -0.7708$. At this transition, an eigenvalue of the Jacobian for $C_{434}$ passes through $-1$, causing the 3-cycle to lose stability and become a saddle of flip type, while simultaneously giving birth to a new stable 6-cycle that orbits the now-unstable points. Throughout this process, the coexisting saddle 3-cycle $C_{334}$ (marked by azure circles and having the symbolic sequence $3 \to 3 \to 4$), which was originally created alongside $C_{434}$ via a fold border collision bifurcation, retains its saddle stability type. Crucially, the stable manifold of this partner saddle, $W^s(C_{334})$, continues to define the boundary of the basin of attraction. Therefore, the bifurcation represents a local topological change within the basin's interior---where the attractor evolves from period-3 to period-6---while the global basin boundary defined by the saddle $C_{334}$ remains robust, confining the new 6-cycle within the same region of phase space.

Alternatively, the 3-cycle can be destroyed via a terminal Border Collision Bifurcation (BCB) when its periodic points collide with the quadrant boundaries. This mechanism is illustrated in Fig.~\ref{fig:3cycle_bcb}. In the left panel, for parameters $(a, c) = (-1, 0.2)$, the 3-cycle undergoes a collision with the border $x=0$ (specifically, the point $\mathbf{p_3}$ hits the boundary satisfying $a^2 + ac - c^2 + c - 1 = 0$). In this configuration, the system exhibits multistability: the disappearing 3-cycle coexists with a stable 4-cycle, whose basin of attraction is shown in azure. A second mechanism occurs in the parameter region near $a \approx -0.39$ (right panel of Fig.~\ref{fig:3cycle_bcb}), where the point $\mathbf{p_1} \in Q_4$ collides with the $x$-axis ($y=0$). This bifurcation happens precisely when the numerator of the $y$-coordinate vanishes ($N_{yp1}(a,c) = a^2 + (3c-2)a + 3c^2 - 5c + 1 = 0$), causing $\mathbf{p_1}$ to cross into Quadrant 1 and destroying the attractor.

\begin{figure}[h!]
	\centering
	\includegraphics[scale=0.6]{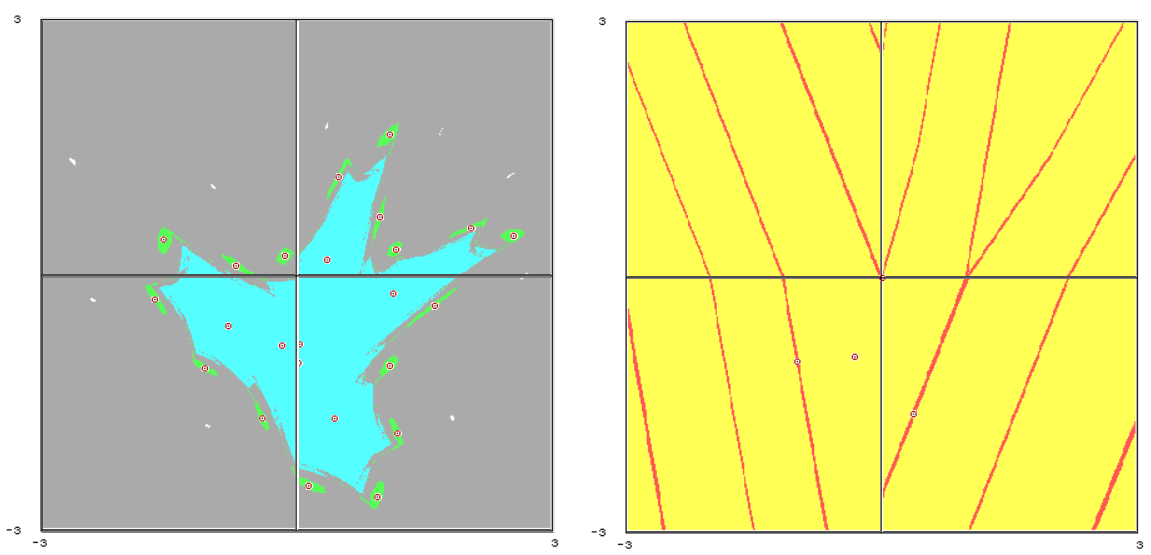} 
	\caption{Terminal Border Collision Bifurcations (BCB) of the 3-cycle $C_{434}$ involving coexisting attractors. \textbf{Left:} At $(a, c) = (-1, 0.2)$, the 3-cycle (red points) collides with the border $x=0$. The coexisting attractor is a stable 4-cycle (azure basin). \textbf{Right:} At $(a, c) \approx (-0.39, 0.383)$, the 3-cycle collides with the border $y=0$ (point $\mathbf{p_1}$), leading to its destruction.}
	\label{fig:3cycle_bcb}
\end{figure}
\subsection{Analysis of  4-Cycles}

Similar to the period-3 case, the period-4 dynamics in this system are organized by Fold Border Collision Bifurcations. The 4-cycles emerge in pairs (one stable, one saddle) when a periodic point collides with a discontinuity boundary.

We identify two distinct attracting branches of period 4, denoted as $C_{1341}$ and $C_{1344}$. Each of these stable cycles is accompanied by a partner saddle cycle—created at the same bifurcation moment—whose symbolic sequence differs by exactly one letter. This structure confirms that the "rich dynamics" observed are not random but strictly governed by the folding mechanism of the map across the critical lines.

In the following propositions, we detail the coordinates and stability regions for the attracting components of these pairs.
\begin{proposition}[Existence of 4-Cycle $C_{1341}$]
	There exists a 4-cycle with symbolic sequence $1 \rightarrow 3 \rightarrow 4 \rightarrow 1$, denoted as $C_{1341} = \{\mathbf{q_1}, \mathbf{q_2}, \mathbf{q_3}, \mathbf{q_4}\}$, where $\mathbf{q_1} \in Q_1$, $\mathbf{q_2} \in Q_3$, $\mathbf{q_3} \in Q_4$, and $\mathbf{q_4} \in Q_1$. The explicit rational coordinates for these periodic points, which share the common denominator $D_q(a,c) = a^4 - 2ca^3 - 2a^2 + (2c^3 + 2c)a - 2c^4 + 2$, are provided in Appendix \ref{app:coordinates}.
\end{proposition}

\begin{proposition}[Stability of 4-Cycle $C_{1341}$]
	The stability of the cycle $C_{1341}$ is determined by the Jacobian $J_{C1} = J_1 \cdot J_4 \cdot J_3 \cdot J_1$. The cycle is locally stable if and only if the following Jury conditions are met:
	\begin{enumerate}
		\item $|(a-c)^3(a+c)| < 1$
		\item $2 - 2a^2 + 2ac - c^4 - (a-c)^3(a+c) > 0$
		\item $2a^2 - 2ac + c^4 - (a-c)^3(a+c) > 0$
	\end{enumerate}
\end{proposition}

\begin{proposition}[Existence of 4-Cycle $C_{1344}$]
	There exists a second 4-cycle with symbolic sequence $1 \rightarrow 3 \rightarrow 4 \rightarrow 4$, denoted as $C_{1344} = \{\mathbf{r_2}, \mathbf{r_3}, \mathbf{r_4}, \mathbf{r_1}\}$, where $\mathbf{r_2} \in Q_1$, $\mathbf{r_3} \in Q_3$, $\mathbf{r_4} \in Q_4$, and $\mathbf{r_1} \in Q_4$. The explicit rational coordinates for these periodic points, which share the common denominator $D_r(a,c) = a^4 - 2(c^2+1)a^2 + 2c^4 + 2$, are provided in Appendix \ref{app:coordinates}.
\end{proposition}

\begin{proposition}[Stability of 4-Cycle $C_{1344}$]
	The stability of the cycle $C_{1344}$ is determined by the Jacobian $J_{C2} = J_4 \cdot J_4 \cdot J_3 \cdot J_1$. The cycle is locally stable if and only if the following conditions are met:
	\begin{enumerate}
		\item $(c^2 - a^2)^2 < 1$
		\item $2 - 2a^2 + c^4 + (c^2 - a^2)^2 > 0$
		\item $2a^2 - c^4 + (c^2 - a^2)^2 > 0$
	\end{enumerate}
\end{proposition}

The stability regions of these 4-cycles are shown in Fig.~\ref{fig:4cycle_combine}. The stability regions of  two  4-cycles are organized by specific bifurcation and existence boundaries defined by the system's parameters. For the cycle $C_{1341}$ in purple region, the stability region is bounded by the \textbf{Purple solid line}, which represents a Center  bifurcation defined by the condition $|\det(J_{1341})| = |(a - c)^3 (a + c)| = 1$, where the cycle loses stability and gives birth to an invariant closed curve. This region is further delimited by existence boundaries where the cycle is destroyed via collision with phase space axes: the \textbf{Orange solid line} ($N_{xq1}=0$) marks the collision of point $\mathbf{q_1}$ with the $x=0$ border, while the \textbf{Red dashed line} ($N_{yq2}=0$) marks the collision of point $\mathbf{q_2}$ with the $y=0$ border. Similarly, for the cycle $C_{1344}$ in green region, stability is bounded by the \textbf{Green solid line}, a Center bifurcation defined by $(c^2 - a^2)^2 = 1$, and the \textbf{Blue dashed line} ($N_{yr2}=0$), which represents the destruction of the cycle via the collision of point $\mathbf{r_2}$ with the $y=0$ border. Finally, these two stability domains are separated by a shared boundary, the \textbf{Black solid line}, which represents the main Border Collision Bifurcation defined by $N_{yq4}=0$; crossing this curve causes the point $\mathbf{q_4}$ (or $\mathbf{r_1}$) to collide with the $y=0$ border, triggering the immediate topological transition between the symbolic sequences $\textit{1341}$ and $\textit{1344}$.

\begin{figure}[h!]
	\centering
	\includegraphics[scale=0.6]{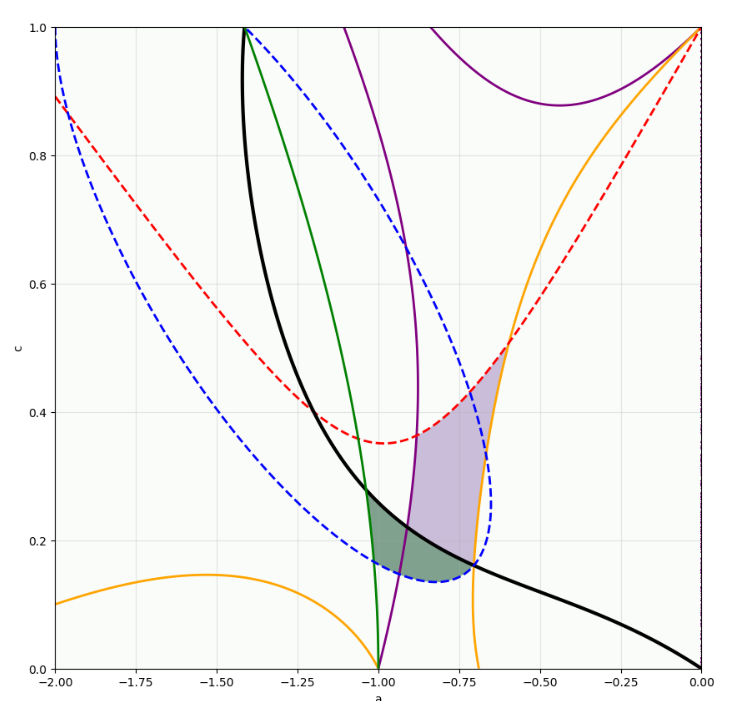} 
	\caption{Stability regions in the parameter plane. The green and purple shaded areas represent the stability regions for the two distinct 4-cycles. The black line represents the Border Collision Bifurcation (BCB) curve; crossing this curve causes a change in the symbolic sequence of the 4-cycle.}
	\label{fig:4cycle_combine}
\end{figure}

We observe significant overlap between the stability regions of the fixed point, the 3-cycle, and these 4-cycles, as in Fig. \ref{fig:4cycle_overlap}  confirming the system's multistability.
\begin{figure}[h!]
	\centering
	\includegraphics[scale=0.5]{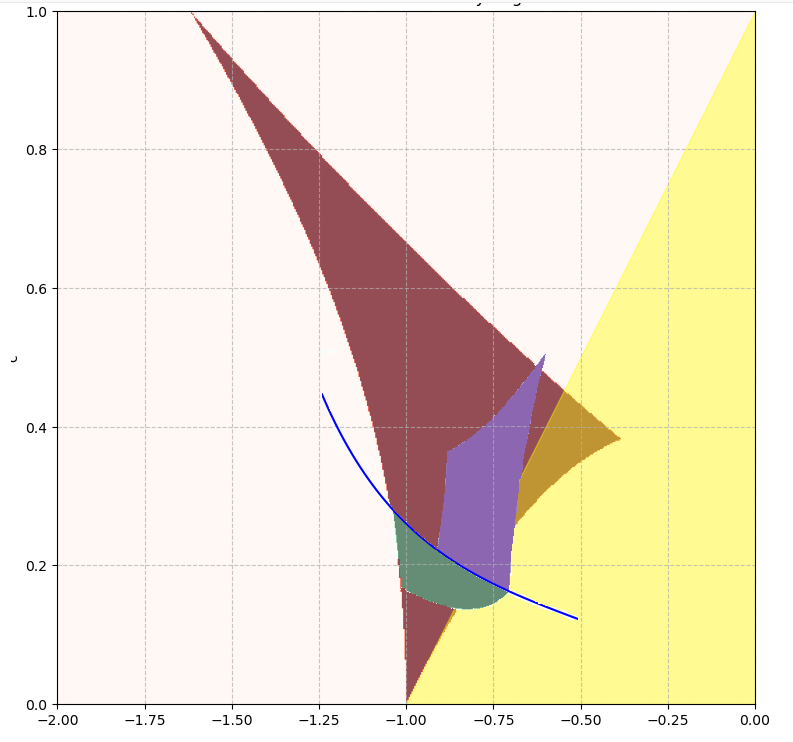} 
	\caption{Region of multistability. In this area of the parameter space, the stability regions (and bifurcation curves) corresponding to the fixed point, the 3-cycle, and the 4-cycles overlap, allowing for coexisting attractors.}
	\label{fig:4cycle_overlap}
\end{figure}

The two 4-cycles are linked via a \textbf{Border Collision Bifurcation}. As illustrated in Fig.~\ref{fig:4cycle_bcb}, crossing the bifurcation curve transforms the symbolic sequence from $\textit{1341}$ to $\textit{1344}$.
\begin{proposition}[Existence of Saddle 4-Cycle $C_{3344}$]
	There exists a 4-cycle with symbolic sequence $3 \rightarrow 3 \rightarrow 4 \rightarrow 4$, denoted as $C_{3344} = \{\mathbf{s_1}, \mathbf{s_2}, \mathbf{s_3}, \mathbf{s_4}\}$, where $\mathbf{s_1}, \mathbf{s_2} \in Q_3$ and $\mathbf{s_3}, \mathbf{s_4} \in Q_4$. The explicit rational coordinates for these periodic points, which share the common denominator $D_{s4}(a,c) = a(a^3 - 2ac^2 - 2a - 4c^2)$, are provided in Appendix \ref{app:coordinates}.
\end{proposition}

The dynamics of the period-4 orbits are organized by a sequence of bifurcations that govern their existence, stability, and topological structure. We focus here on the transition between the two stable symbolic sequences and the origin of the saddle cycle that organizes the global basin structure.

The two stable 4-cycles, $C_{1341}$ and $C_{1344}$, are topologically distinct but dynamically linked via a BCB. As illustrated in Fig.~\ref{fig:4cycle_bcb}, this transition is triggered when the periodic point $\mathbf{q_4} \in Q_1$ of the cycle $C_{1341}$ collides with the border $y=0$.
\begin{itemize}
	\item \textbf{Before the Bifurcation:} The cycle follows the sequence $1 \to 3 \to 4 \to 1$. The point $\mathbf{q_4}$ has a positive $y$-coordinate.
	\item \textbf{The Collision:} As parameters cross the curve $N_{yq4}(a,c) = 0$ (the black curve in our stability diagrams), the $y$-coordinate of $\mathbf{q_4}$ vanishes.
	\item \textbf{After the Bifurcation:} The point crosses into Quadrant 4, becoming the point $\mathbf{r_1}$ of the new cycle. The symbolic sequence changes from $\dots \to 4 \to 1$ to $\dots \to 4 \to 4$, resulting in the sequence $1 \to 3 \to 4 \to 4$ (or $C_{1344}$).
\end{itemize}
This mechanism ensures a continuous transition between the two attractors, where the "Old" cycle is destroyed and the "New" cycle is immediately created with the same stability properties. The figure ~\ref{fig:4cycle_bcb} illustrates this transition numerically. In the \textbf{left panel}, at parameters $(a, c) = (-0.85, 0.2)$, the system exhibits a stable 4-cycle with symbolic sequence $\textit{1341}$, whose basin of attraction is shown in azure. In the \textbf{right panel}, at $(a, c) = (-0.86, 0.2)$, the parameters have crossed the BCB curve. The attractor has seamlessly transformed into a 4-cycle with symbolic sequence $\textit{1344}$ (basin in azure), confirming the theoretical prediction of the border collision mechanism.

The phase space also contains a saddle 4-cycle, denoted as $C_{3344}$ (with sequence $3 \to 3 \to 4 \to 4$), which plays a critical role in the global dynamics. This saddle cycle emerges via a \textbf{Fold Border Collision Bifurcation} that marks the onset of the period-4 resonance region.
\begin{itemize}
	\item \textbf{Mechanism:} In a typical Fold-BCB scenario for piecewise linear maps, a pair of cycles (one stable, one unstable) is created simultaneously when a periodic point collides with a border. The saddle $C_{3344}$ is the unstable partner created alongside the stable 4-cycle branch.
	\item \textbf{Basin Boundary Role:} As observed numerically, the periodic points of this saddle cycle lie on the boundary separating the basin of attraction of the stable 3-cycle ($C_{434}$) from the basin of the stable 4-cycle ($C_{1341}$ or $C_{1344}$). The stable manifold of this saddle, $W^s(C_{3344})$, forms the separatrix between these two coexisting attractors, organizing the multi-stable phase space.
\end{itemize}

 Additionally, the $C_{1344}$ cycle undergoes a \textbf{Center Bifurcation}, where the stable cycle is replaced by an invariant closed curve (see Fig.~\ref{fig:4cycle_center}).
\begin{figure}[h!]
	\centering
	\includegraphics[scale=0.55]{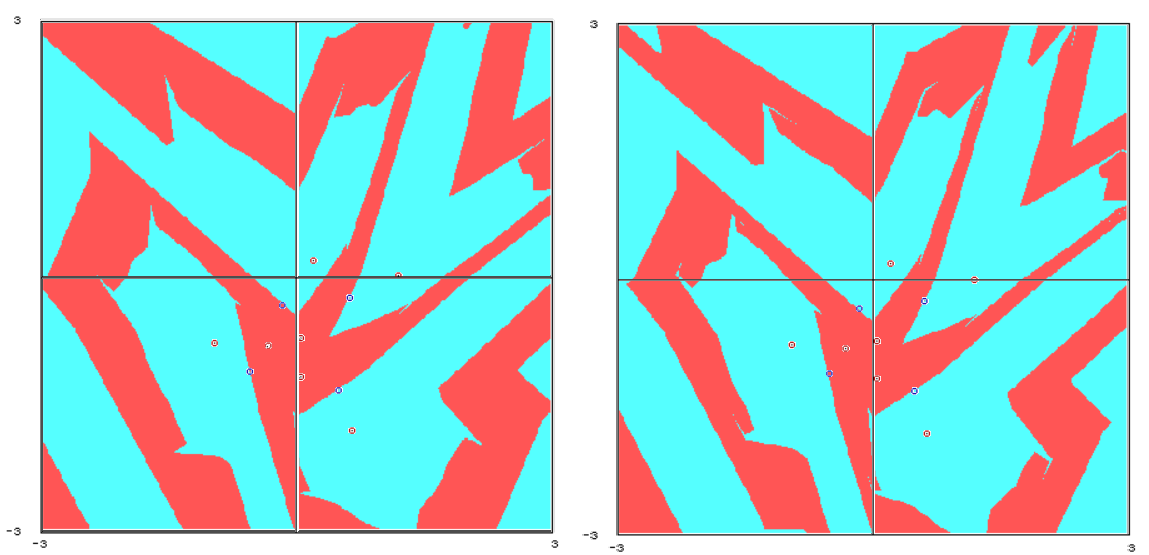} 
	\caption{Coexistence of periodic cycles and the Border Collision Bifurcation. \textbf{Left:} At $(a, c) = (-0.85, 0.2)$, a stable 4-cycle with symbolic sequence $\textit{1341}$ exists. \textbf{Right:} At $(a, c) = (-0.86, 0.2)$, after crossing the BCB, the attractor becomes a 4-cycle with symbolic sequence $\textit{1344}$. A stable 3-cycle coexists in both scenarios, with the basin of attraction of the 3-cycle in red and the basin of attraction of the 4-cycle in azure.}
	\label{fig:4cycle_bcb}
\end{figure}

\begin{figure}[h!]
	\centering
	\includegraphics[scale=0.45]{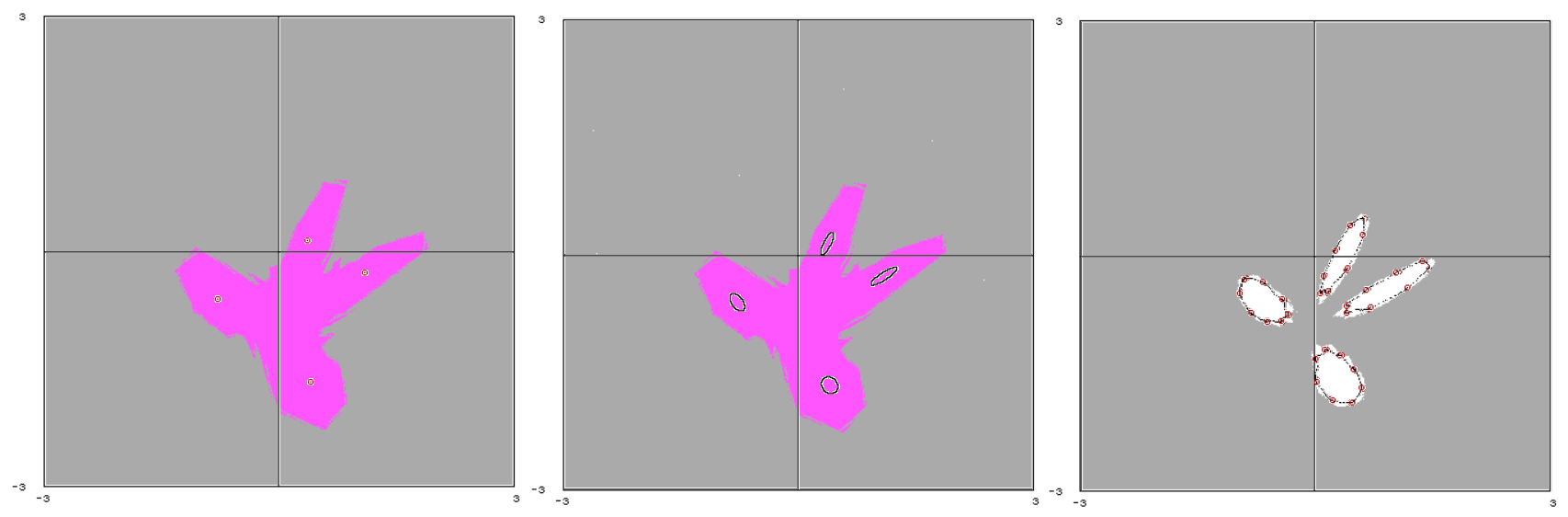} 
	\caption{Center bifurcation of the 4-cycle $C_{1344}$ for fixed $c = 0.2$. \textbf{Left:} Before the bifurcation ($a = -1.019$), a stable 4-cycle exists. \textbf{Middle:} At $a = -1.0199$, the 4-cycle is destroyed and replaced by a stable invariant closed curve. \textbf{Right:} At $a = -1.06$, the invariant curve persists, but global instability increases as other orbits diverge.}
	\label{fig:4cycle_center}
\end{figure}

\section{Investigation of the First Return Map in the Degenerate Case $c=a$}

Following the methodology established in \cite{GardiniTikjha2020CSF}, we investigate the global dynamics of the map $T$ in the degenerate parameter regime where $c=a$. In this case, the Jacobian determinants for the linear maps in quadrants $Q_1$ and $Q_3$ vanish ($\det(J_1) = \det(J_3) = c-a = 0$), implying that these two-dimensional regions are mapped onto one-dimensional subsets of the phase space.

\subsection{Collapse of Partitions and Critical Lines}
Substituting $c=a$ into the map definition, we derive the images of the degenerate partitions:

\begin{proposition}
	For $c=a$, the map $T$ collapses the first and third quadrants onto specific lines:
	\begin{enumerate}
		\item The image of $Q_1$ ($x \ge 0, y \ge 0$) lies entirely on the line $L_1: y = x$.
		\item The image of $Q_3$ ($x \le 0, y \le 0$) lies entirely on the line $L_2: y = -x - 2$.
	\end{enumerate}
\end{proposition}

\begin{proof}
	For any point $(x,y) \in Q_1$, the map is $T_1(x,y) = (x+ay-1, x+ay-1)$. Since the $x$ and $y$ components are identical, the image lies on the line $L_1: y=x$.
	For any point $(x,y) \in Q_3$, the map is $T_3(x,y) = (-x+ay-1, x-ay-1)$. Let $x' = -x+ay-1$. The $y$-component is $y' = -( -x+ay ) - 1 = -(x'+1) - 1 = -x' - 2$. Thus, the image lies on the line $L_2: y=-x-2$.
\end{proof}

\subsection{Derivation of the 1D First Return Map}
Since any trajectory entering $Q_1$ or $Q_3$ is immediately mapped onto the set $\Lambda = L_1 \cup L_2$, the long-term dynamics are governed by the restriction of the map to these lines. We construct the First Return Map $F(t)$ on the critical line $L_1$, where $t$ represents the coordinate along the line such that a point $P \in L_1$ is given by $P = (t, t)$.

\paragraph{Branch 1: Direct Return ($t \ge 0$)}
For a point $P = (t, t) \in L_1$ with $t \ge 0$, the point lies in $Q_1$. The map $T_1$ applies, and the image remains on $L_1$.
The new coordinate $t'$ is given by the $x$-component of the image:
\[ t' = x' = t + at - 1 = (1+a)t - 1 \]
Thus, the first branch of the return map is:
\[ F(t) = (1+a)t - 1 \quad \text{for } t \ge 0 \]

\paragraph{Branch 2: Indirect Return via $L_2$ ($t < 0$)}
For a point $P = (t, t) \in L_1$ with $t < 0$, the point lies in $Q_3$. The map $T_3$ applies, mapping the point to $L_2$.
The intermediate point $P_{int} \in L_2$ has coordinate $x_{int}$:
\[ x_{int} = -t + at - 1 = (a-1)t - 1 \]
From this point $P_{int} = (x_{int}, -x_{int}-2)$, the trajectory may enter $Q_4$ (if $x_{int} \ge 0$) or $Q_2$ (if $x_{int} \le -2$). Assuming the trajectory enters $Q_4$ and subsequently $Q_1$ (a typical sequence for the global attractor), we apply $T_4$ followed by the return to $L_1$. The explicit form of the return map for this sequence is derived by composing the linear functions:
\[ t' = -(1+a)(a-1)^2 t - (a^2 + 3a + 3) \]
This defines the second branch of the map for $t < 0$, provided the intermediate iterations remain in the valid partitions.

\begin{corollary}[1D Map Definition]
	The dynamics on the attractor are governed by the piecewise linear one-dimensional map $F: \mathbb{R} \to \mathbb{R}$:
	\[ F(t) = \begin{cases} 
		(1+a)t - 1 & \text{if } t \ge 0 \\
		-(1+a)(a-1)^2 t - (a^2 + 3a + 3) & \text{if } t < 0
	\end{cases} \]
\end{corollary}

\begin{remark}[Discontinuity and Chaotic Dynamics]
	It is important to note that the derived map $F(t)$ is \textbf{discontinuous} at the critical point $t=0$. The limit from the right is $\lim_{t \to 0^+} F(t) = -1$, while the limit from the left is $\lim_{t \to 0^-} F(t) = -(a^2 + 3a + 3)$.
	
	Regarding the global dynamics, we observe that for $a > 0$, the slope of the right branch is $1+a > 1$, making the map locally expanding. Numerical evidence confirms that for $a > 0$, this expansivity combined with the folding action of the discontinuity leads to a \textbf{chaotic interval}. An example of this discontinuous map and the resulting chaotic behavior is illustrated in Fig.~\ref{fig:1DMap} for $a=0.5$.
\end{remark}

\begin{figure}[h]
	\centering
	 \includegraphics[scale=0.25]{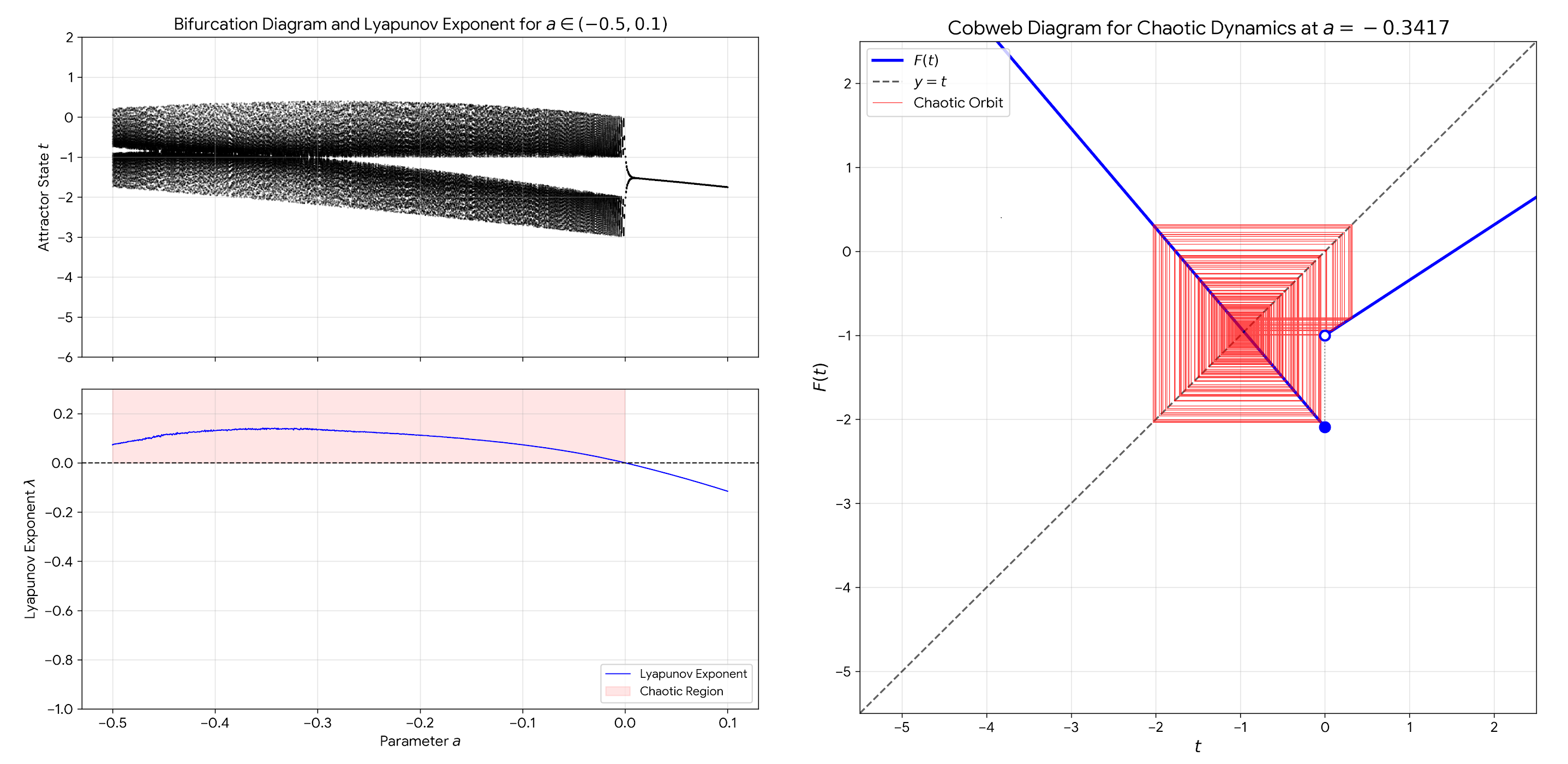} 
	\caption{Numerical analysis of chaos in the negative parameter regime. \textbf{Left Panel:} Superimposed bifurcation diagram (black points) and Lyapunov exponent (blue curve) for $a \in (-0.5, 0.1)$. The parameter regions where the Lyapunov exponent is strictly positive ($\lambda > 0$, shaded red) align perfectly with the dense chaotic bands, confirming the existence of chaotic attractors. \textbf{Right Panel:} Cobweb diagram for the specific chaotic parameter $a = -0.3417$ with initial condition $t_0 = 0.1$. }\label{fig:1DMap}
\end{figure}

\begin{proposition}[Invariant Interval]
	The global dynamics of the 1D map $F(t)$ are confined within an absorbing interval $I = [F(0^-), \mathcal{B}]$, where the upper bound $\mathcal{B}$ depends on the parameter $a$:
	\begin{itemize}
		\item \textbf{Case $a < 0$ (Chaotic Regime):} The map is expanding, and the interval extends to the second iterate of the discontinuity.
		\[ I = [ F(0^-), \ F^2(0^-) ] = [ -(a^2 + 3a + 3), \ F(F(0^-)) ] \]
		\item \textbf{Case $a > 0$ (Stable Regime):} The map is contracting, and the interval is typically bounded by the immediate image of the discontinuity from the right.
		\[ I = [ F(0^-), \ -1 ] = [ -(a^2 + 3a + 3), \ -1 ] \]
	\end{itemize}
	Trajectories starting outside this interval are eventually mapped into it and cannot escape.
\end{proposition}

\begin{example}[Absorbing Interval in the Chaotic Regime]
	To illustrate the analytical bounds established in Proposition 4.4, let us consider the one-dimensional return map $F(t)$ at the parameter $a = -0.3417$, which corresponds to a robust chaotic regime as depicted in Figure 11.
\end{example}

		Substituting $a = -0.3417$ into the map derived in Corollary 4.2 yields the following piecewise linear function (with coefficients rounded to four decimal places):
		\begin{equation}
			F(t) =
			\begin{cases}
				0.6583t - 1, & \text{if } t \ge 0, \\
				-1.1850t - 2.0917, & \text{if } t < 0.
			\end{cases}
		\end{equation} 

		Because $a < 0$, the system is locally expanding along the negative branch, and the invariant absorbing interval $I$ is determined by the first and second iterates of the left-sided limit at the discontinuity, $I = [F(0^-), F^2(0^-)]$. Evaluating these boundaries, we find:
		\begin{align*}
			F(0^-) &= -(a^2 + 3a + 3) \approx -2.0917, \\
			F^2(0^-) &= F(-2.0917) = -1.1850(-2.0917) - 2.0917 \approx 0.3870.
		\end{align*}
		Thus, the rigorous absorbing interval is $I = [-2.0917, 0.3870]$.
		
		To demonstrate the absorbing nature of this interval, consider two initial conditions originating outside of $I$:
		\begin{enumerate}
			\item For an initial point above the upper bound, let $t_0 = 2$. Applying the positive contracting branch of the map, we obtain $t_1 = F(2) = 0.6583(2) - 1 = 0.3166$. In a single iteration, the trajectory maps strictly into the interval $I$.
			\item For an initial point below the lower bound, let $t_0 = -3$. Applying the negative expanding branch, the point is mapped to $t_1 = F(-3) = -1.1850(-3) - 2.0917 = 1.4633$. Since $t_1$ is positive but still outside $I$, we apply the positive branch for the subsequent iteration, yielding $t_2 = F(1.4633) = 0.6583(1.4633) - 1 = -0.0367$. By the second iteration, the trajectory is permanently captured within $I$.
		\end{enumerate} 
		Once trajectories enter $I$, the interplay between the local expansivity and the global reinjection mechanism at the discontinuity sustains the chaotic dynamics confined entirely within these bounds.

\begin{remark}[Rigorous Boundedness versus Numerical Chaos]
	While Proposition 4.4 rigorously proves the existence of an absorbing invariant interval $I$, demonstrating that global trajectories are analytically bounded, the existence of robust chaos within this interval relies on a combination of analytical properties and numerical observation. The local expansivity for $a > 0$ establishes sensitive dependence on initial conditions; however, the precise topological structure and density of the chaotic attractors illustrated in the bifurcation diagrams (e.g., Fig. 11) are currently verified numerically. The map exhibits a complex balance between expanding ($t \ge 0$) and contracting ($t < 0$) branches, meaning the existence of topological chaos is parameter-dependent and mixed with periodic windows.
\end{remark}
\section{Conclusion and Discussion}

In this work, we have presented a comprehensive analysis of the global dynamics of a two-dimensional piecewise linear map characterized by four partitions. The model, defined by absolute value terms for both state variables ($|x|$ and $|y|$), exhibits a rich bifurcation structure that significantly extends the behaviors observed in standard two-partition maps.

Our investigation of periodic orbits established that the primary route to period-doubling is initiated by a Flip bifurcation of the fixed point $P_3$ at $a=0$. A central finding of this study is the unification of the period-2 dynamics into a single branch that manifests in two topologically distinct configurations, $C_{34}$ and $C_{24}$. We demonstrated that these cycles never coexist; instead, they are linked via a Border Collision Bifurcation (BCB) that occurs precisely when a periodic point collides with the discontinuity boundary.

For higher periodicities, we showed that stable cycles of period 3 and 4 emerge via Fold Border Collision Bifurcations. Consistent with the theory of discontinuous maps, these attractors appear in pairs consisting of one stable cycle and one saddle cycle. We analytically derived the boundaries for these bifurcations and emphasized the critical role of the partner saddle cycles (e.g., $C_{334}$ and $C_{3344}$). The stable manifolds of these saddles act as separatrices, organizing the complex basins of attraction for coexisting multistable attractors.

The analysis of this four-partition map highlights distinct dynamical features, particularly a higher density of coexisting attractors compared to standard two-partition maps, driven by the complex boundaries of saddle manifolds rather than fractal homoclinic structures. Our study of the degenerate regime ($c=a$) rigorously bridges two-dimensional and one-dimensional dynamics, demonstrating that while dimension reduction explains global boundedness, local expansion alone is insufficient to guarantee chaos—as evidenced by stable windows amidst chaos—whereas the robust chaotic regime identified for negative parameters provides a clear mechanism of local expansion coupled with global reinjection.

Future work will focus on extending this analysis to $n$-dimensional generalizations and investigating whether the "weird quasiperiodic attractors" observed in recent literature can exist in this specific four-partition configuration.
\section*{Acknowledgements}
The author was supported by the National Research Council of Thailand and Pibulsongkram Rajabhat University.
\section*{Use of AI tools declaration}
The author acknowledges the use of Artificial Intelligence (Large Language Models) for linguistic editing and proofreading during the preparation of this manuscript. The author assumes full responsibility for the content of the paper.
	\appendix
	\section{Analytical Coordinates of Higher-Period Cycles} \label{app:coordinates}
	
	\subsection{Coordinates of the 3-Cycle $C_{434}$}
	The exact coordinates $p_i = (N_{xpi}/D_p, N_{ypi}/D_p)$ for the periodic points of $C_{434}$ are given by:
	\begin{itemize}
		\item Point 1 ($p_1 \in Q_4$):
		\begin{align*}
			N_{xp1}(a,c) &= (a+c+1)(a^2+(c+2)a+c^2-c+1) \\
			N_{yp1}(a,c) &= a^2+(3c-2)a+3c^2-5c+1
		\end{align*}
			\item Point 2 ( $p_2 \in Q_3$):
		\begin{align*}
			N_{xp2}(a,c) &= (a+c+1)(a^2 + 3ac + 3c^2 - 3c + 3) \\
			N_{yp2}(a,c) &= 3a^2 + (7c+2)a + 5c^2 - c + 3
		\end{align*}
		
		\item Point 3 ($p_3 \in Q_4$):
		\begin{align*}
			N_{xp3}(a,c) &= \left(a+c+1\right) \left(a^{2}+a c-c^{2}+c-1\right) \\
			N_{yp3}(a,c) &= a^2 + (c+2)a + c^2-3c+5
		\end{align*}
	\end{itemize}
\subsection{Coordinates of the Saddle 3-Cycle $C_{334}$}
The exact coordinates $\mathbf{s_i} = (N_{xsi}/D_s, N_{ysi}/D_s)$ for the periodic points of the saddle cycle $C_{334}$ are given by:
\begin{itemize}
	\item \textbf{Point 1 ($\mathbf{s_1} \in Q_3$):}
	\begin{align*}
		N_{xs1} &= (a + c + 1)(a^2 + ac + 2a + c^2 - c + 1) \\
		N_{ys1} &= 3a^2 + 3ac + 4a + c^2 + c + 1
	\end{align*}
	\item \textbf{Point 2 ($\mathbf{s_2} \in Q_3$):}
	\begin{align*}
		N_{xs2} &= (a + c + 1)(a^2 + ac - c^2 + c - 1) \\
		N_{ys2} &= 3a^2 + ac + 4a - c^2 - c + 1
	\end{align*}
	\item \textbf{Point 3 ($\mathbf{s_3} \in Q_4$):}
	\begin{align*}
		N_{xs3} &= (a + c + 1)(a^2 - ac + 2a + c^2 - c + 1) \\
		N_{ys3} &= a^2 + ac + c^2 - c - 1
	\end{align*}
\end{itemize}
\subsection{Coordinates of the 4-Cycle $C_{1341}$}
The exact coordinates $\mathbf{q_i} = (N_{xqi}/D_q, N_{yqi}/D_q)$ for the periodic points of the cycle $C_{1341}$ are given by:
\begin{itemize}
	\item \textbf{Point 1 ($\mathbf{q_1} \in Q_1$):}
	\begin{align*}
		N_{xq1} &= a^4 + (3 - 3c)a^3 - 2c^2a^2 + (2c^3 - 4c^2 + 2c - 4)a + (2c^4 - 2) \\
		N_{yq1} &= a^3 + (2c - 2)a^2 + (-2c - 4)a + (-4c^3 + 6c^2 - 2)
	\end{align*}
	\item \textbf{Point 2 ($\mathbf{q_2} \in Q_3$):}
	\begin{align*}
		N_{xq2} &= a^4 + (c+1)a^3 - (2c^2 + 2c + 2)a^2 + (-4c^3 + 2c^2 - 6)a + (4c^4 - 4) \\
		N_{yq2} &= 3a^3 + (2 - 2c)a^2 - (6c^2 + 4c + 4)a + (6c^3 - 2c - 4)
	\end{align*}
	\item \textbf{Point 3 ($\mathbf{q_3} \in Q_4$):}
	\begin{align*}
		N_{xq3} &= a^4 + (1 - c)a^3 - (4c^2 + 2c)a^2 + (8c^3 - 2c^2 - 4c + 2)a + (2 - 2c^4) \\
		N_{yq3} &= a^3 - 4ca^2 + (6c^2 + 2c - 6)a + (2c^2 + 4c - 6)
	\end{align*}
	\item \textbf{Point 4 ($\mathbf{q_4} \in Q_1$):}
	\begin{align*}
		N_{xq4} &= a^4 + (1 - 3c)a^3 + (2c^2 - 4)a^2 + (6c^3 - 2c - 4)a \\
		N_{yq4} &= a^3 + (2 - 2c)a^2 + (2 - 4c^2)a + (-2c^3 - 4c^2 + 6c)
	\end{align*}
\end{itemize}
\subsection{Coordinates of the 4-Cycle $C_{1344}$}
The exact coordinates $\mathbf{r_i} = (N_{xri}/D_r, N_{yri}/D_r)$ for the periodic points of the cycle $C_{1344}$ are given by:
\begin{itemize}
	\item \textbf{Point 1 ($\mathbf{r_2} \in Q_1$):}
	\begin{align*}
		N_{xr2} &= a^4 + (3-c)a^3 + (-6c^2+6c)a^2 + (-6c^3+2c^2+4c-4)a - 2c^4-2 \\
		N_{yr2} &= a^3 + (4c-2)a^2 + (6c^2-4c-4)a + 4c^3-6c^2-2
	\end{align*}
	\item \textbf{Point 2 ($\mathbf{r_3} \in Q_3$):}
	\begin{align*}
		N_{xr3} &= a^4 + (3c+1)a^3 + (2c^2+2c-2)a^2 + (-2c^3-4c^2+4c-6)a - 4c^4-4 \\
		N_{yr3} &= 3a^3 + (4c+2)a^2 + (-2c^2-4)a - 6c^3-2c-4
	\end{align*}
	\item \textbf{Point 3 ($\mathbf{r_4} \in Q_4$):}
	\begin{align*}
		N_{xr4} &= a^4 + (c+1)a^3 - (2c^2+2c)a^2 + (-4c^3+4c^2-6c+2)a + 2c^4+2 \\
		N_{yr4} &= a^3 + (-4c^2+8c-6)a + 2c^2+4c-6
	\end{align*}
	\item \textbf{Point 4 ($\mathbf{r_1} \in Q_4$):}
	\begin{align*}
		N_{xr1} &= a^4 + (c+1)a^3 + (-4c^2+6c-4)a^2 + (-4c^3+6c^2-2c-4)a \\
		N_{yr1} &= a^3 + (2-2c)a^2 + (-4c^2+2)a + (-2c^3-4c^2+6c)
	\end{align*}
\end{itemize}
\subsection{Coordinates of the Saddle 4-Cycle $C_{3344}$}
The exact coordinates $\mathbf{s_i} = (N_{xsi}/D_{s4}, N_{ysi}/D_{s4})$ for the periodic points of the saddle cycle $C_{3344}$ are given by:
\begin{itemize}
	\item \textbf{Point 1 ($\mathbf{s_1} \in Q_3$):}
	\begin{align*}
		N_{xs1} &= (a + c + 1)(a^3 + 2a^2c + 2a^2 + 2ac^2 + 4ac - 2a + 2c^3 - 2c^2 + 2c - 2) \\
		N_{ys1} &= 3a^3 + 6a^2c + 4a^2 + 4ac^2 + 10ac - 2a + 2c^3 + 2c^2 + 2c - 2
	\end{align*}
	\item \textbf{Point 2 ($\mathbf{s_2} \in Q_3$):}
	\begin{align*}
		N_{xs2} &= (a + c + 1)(a^3 + 2a^2c - 2c^3 + 2c^2 - 2c + 2) \\
		N_{ys2} &= 3a^3 + 4a^2c + 2a^2 - 2ac^2 + 6ac - 4a - 2c^3 - 2c^2 + 2c - 2
	\end{align*}
	\item \textbf{Point 3 ($\mathbf{s_3} \in Q_4$):}
	\begin{align*}
		N_{xs3} &= (a + c + 1)(a^3 - 2ac^2 + 4ac - 2a + 2c^3 - 2c^2 + 2c - 2) \\
		N_{ys3} &= a^3 + 2a^2 + 2ac + 2a + 2c^3 - 2c^2 + 2c + 2
	\end{align*}
	\item \textbf{Point 4 ($\mathbf{s_4} \in Q_4$):}
	\begin{align*}
		N_{xs4} &= (a + c + 1)(a^3 + 2a^2 + 4ac + 2c^3 - 2c^2 + 2c - 2) \\
		N_{ys4} &= a^3 + 2a^2c + 2ac^2 + 2ac - 4a + 2c^3 - 2c^2 - 2c - 2
	\end{align*}
\end{itemize}

\end{document}